\documentclass[11pt]{article}
\usepackage[intlimits]{amsmath}
\usepackage{amsfonts, amssymb, amsthm}
\usepackage{esint}
\usepackage{epsfig}
\usepackage{color}
\usepackage{threeparttable} 
\usepackage[toc,page]{appendix}

\renewcommand{\O}{\Omega}

\renewcommand{\a}{\alpha}

\newcommand{\e}{\epsilon}

\renewcommand{\L}{\Lambda}

\renewcommand{\l}{\lambda}
\renewcommand{\d}{\delta}

\newcommand{\s}{\sigma}

\newtheorem{thm}{Theorem}[section]
\newtheorem{prop}[thm]{Proposition}

\newtheorem{lemma}[thm]{Lemma}

\newtheorem{preremark}[thm]{Remark}

\numberwithin{equation}{section}

\newcommand{\norm}[1]{\left\Vert#1\right\Vert}

\newcommand{\R}{\mathbb R}

\newcommand{\grad} {\nabla}

\DeclareMathOperator*{\osc}{osc}

\def\XXint#1#2#3{{\setbox0=\hbox{$#1{#2#3}{\int}$}
       \vcenter{\hbox{$#2#3$}}\kern-.5\wd0}}

\newcommand{\meanbar}[1]{%
\setbox0 = \hbox{$#1 \int$}
\hbox to 0pt{%
\thinspace
\hskip 0.1\wd0
\raise 0.5\ht0
\hbox{%
\lower 0.5\dp0
\hbox{\rule{0.8\wd0}{2\linethickness}}
}%
\hss
}%
}

    \newcounter{myfootertablecounter}

\begin{document}

\title{The Two-Phase Stefan Problem with Anomalous Diffusion} 


\author{Ioannis Athanasopoulos, Luis Caffarelli, Emmanouil Milakis}
\date{}
\maketitle


\begin{abstract}
The non-local in space two-phase Stefan problem (a prototype in phase change problems) can be formulated via a singular nonlinear parabolic integro-differential equation which admits a unique weak solution. This formulation makes Stefan problem to be part of the General Filtration Problems; a class which includes the Porous Medium Equation. In this work, we prove that the weak solutions to both Stefan and Porous Media problems are continuous.
\end{abstract}



\section{Introduction}\label{intr}
In this paper we study the parabolic, non-local in space, initial-boundary value problems i.e.
\begin{align}
\begin{cases} \label{eq1.1}
 \beta_t(u(x,t))\ni{L} u(x,t) \ \ & (x,t)\in Q:=\Omega \times (0,T] \\
u(x,t)=0 \ \ & (x,t)\in \Omega^c\times (0,T)] \\
 u(x,0)=u_0 (x) \  \  & x\in \Omega\times\{0\}
 \end{cases}
\end{align}
where
$${L}u(x,t)=\int_{\R^n}(u(y,t)-u(x,t))K(x,y,t)dy$$
with $K$ symmetric in $x$ and $y$ i.e. $K(x,y,t)=K(y,x,t)$ for any $x\neq y$ and satisfying
$$\frac{\mathbf{1}_{\{|x-y|\leq 2\}}}{\Lambda}\frac{1}{|x-y|^{n+\alpha}}\leq K(x,y.t)\leq \frac{\Lambda}{|x-y|^{n+\alpha}}$$
for some $\Lambda>0$ and $\alpha\in(0,2)$
and  
$\beta$ is either (i) a monotone graph
\begin{align}\beta(x)= \begin{cases} \label{eq1.2} 
ax-1, & x<0\ (a>0) \cr [-1,+1], & x=0 \cr bx+1, & x>0\ (b>0) 
\end{cases}
\end{align}
or (ii) a continuous increasing real-valued function satisfying:
\begin{align}
\begin{cases} \label{eq1.3}
(a) \ \  & \beta'(x)\  \text{exists for all}\ x\neq 0 \\
(b) \ \ & \beta(0)=0 \\
(c) \ \ & \beta'(x)\geq c_1>0 \\
(d) \ \ & \beta'(x)\leq C(\epsilon)\ \text{for}\ x \in (-\frac{1}{\epsilon},-\epsilon)\cup (\epsilon,\frac{1}{\epsilon}) \ \ \text{and} \ \  \epsilon >0.
\end{cases}
\end{align}

Equations (\ref{eq1.1}) and (\ref{eq1.2}) describe the flow of heat within a substance which changes phase at temperature zero and equations (\ref{eq1.1}) and (\ref{eq1.3}) includes the porous media equation.
The main result in this work (Theorem \ref{mth}) asserts  that the solution $u$ is a continuous function of $x$ and $t$ whose modulus of continuity depends on $\beta$. In case (ii), if, in addition, we assume that $\beta$ near zero has homogeneous behavior as that of porous media, i.e. $\beta(u)  \sim   u^{1/m} , m>1$, then we obtain H\"older modulus of continuity. 

In \S\ref{prel} we give the definition of weak solutions and list some symbols used throughout this paper. In \S\ref{exist} we construct solutions $u^{\varepsilon}$ to appropriate approximate problems which yield continuity, i.e.  $u^{\varepsilon}$ are defined pointwise. In order to obtain a oscillation decay we consider two alternatives: one is when on average $u^\varepsilon$ is very close to the singularity of $\beta$ and the second when it is far from it. The main estimate i.e. Lemma \ref{mlemma} takes care the former case. Lemma \ref{altl} of \S\ref{osc} handles the latter, more delicate alternative. By iteration we obtain in \S\ref{cont} the continuity of the approximate solutions and consequently our reult.

The classical two-phase Stefan problem was treated in \cite{CE} and the fractional two-phase Stefan problem in \cite{AC}. In \cite{AC} the treatment relied on the extension to an additional one dimension (see \cite{CS}). In this work, the kernel we introduce does not allow any form of extension and therefore the approach is direct and more general. In fact, the result obtained here, generalizes the fractional as well as the classical cases. In a forthcoming paper we shall treat the case where the non-locality occurs in space and time simultaneously, such as in masters kernels, as it was done for the fractional case in \cite{AC}; thus completing the generalization of the results in [1].

\section{Preliminaries}\label{prel}

In the present paper we consider evolution equations involving non-local operators. When we say that an equation of the form
$$\partial_t w(x,t)=\int_{\R^n}[w(y,t)-w(x,t)]K(x,y,t)dy$$
is satisfied in the weak sense in $Q$, we mean that for every smooth test function $\eta$, which is vanishing on the parabolic boundary of $Q:=\Omega\times (0,T]$ and at $\{t=T\}$, we have
$$\int_{0}^T\bigg[\int_{\R^n}\partial_t w(x,t)\cdot\eta(x,t)dx+D(w,\eta)\bigg]dt=0
$$
where
$$D(u,v):=\int_{\R^n}\int_{\R^n}(u(x,t)-u(y,t))K(x,y,t)((v(x,t)-v(y,t))dydx.$$
Similarly, in equation (\ref{eq1.1}), when we write 
 $$\beta_t(u(x,t))\ni{L} u(x,t) \ \ \text{for} \ (x,t)\in Q$$
 we mean that 
 $$\int_{0}^T\bigg[\int_{\R^n}\partial_tv(x,t)\eta(x,t)dx+D(u,\eta)\bigg]dt=0
$$
for some $v\in\beta(u)$ a.e. and all test function $\eta$ vanishing in the parabolic boundary and at $\{t=T\}$. Note also the extra condition in (\ref{eq1.1}), $u(x,t)$ is always assumed to be zero in $\Omega^c\times (0,T)]$. \\

\textbf{Notation}s:
$$\osc_D  u:=\max_D u-\min_D u\ \ \ \ \  \ \ \ \ \ \ \ \ \ \ \ \ \ \ \ \ \ \ \ \  \ \ \ \ \ \ \ \ \ \ \ \ \ \ \ \ \ \ \ \ \ \ \ \ \ \ \ \ \ \ \  \ \ \ \ \ \ \ \ \ \ \ \ \ \ \ \ \ \ \ \ \ \ \ \ \ \ \ \ \ \ \  \ \ $$

$\grad u=(u_{x_1},...,u_{x_n})$ denotes the spatial gradient of $u$.

$Q:=\O\times (0,T),\ \ \O\subset \R^n $

$Q_1\equiv Q_1(0,0),\ \ Q_R\equiv Q_R(0,0)$ 

$Q_R(x_0,t_0):=\{(x,t): |x-x_0|<R, \  t_0-R^\a<t<t_0\}$

$B_R(x_0)$ is the open ball of radius $R$ centered at $x_0$.

$|A|$ is the Lebesgue measure of the set $A$.

$\fint_{A}f dx=\frac{1}{|A|}\int_A f dx$ is the average of $f$ over $A$.

$u^+=\max\{u,0\}$, $u^-=-\min\{u,0\}$

$\mathbf{1}_A$ is the characteristic function of the set $A$.\\
Although we use the same letter for a universal constant, its precise value can vary from line to line.

\section{Existence of the approximate problems}\label{exist}
If we write $\beta:=\phi^{-1}$ then  (\ref{eq1.1}) can be expressed by
\begin{align}
\begin{cases} \label{eq3.1}
\partial_t v-L\phi(v)=0 \ \ & (x,t)\in Q:=\Omega \times (0,T] \\
v(x,t)=0 \ \ & (x,t)\in \Omega^c\times (0,T)] \\
 v(x,0)=v_0 (x) \  \  & x\in \Omega\times\{0\}.
 \end{cases}
\end{align}

 Formulation (\ref{eq3.1}) represents a class of nonlinear degenerate diffusion problems called Generalized Filtration Equations (see \cite{TEV}) which includes the Stefan Problem, the Porous Media Problem, and the Fast Diffusion Problem. Weak solutions can be constructed using the theory of \cite{BB}, in the general framework of proving existence of weak solutions to problems which involve maximal monotone operators in real Hilbert spaces. 
 
The regularized problem for our specific choice of $\beta$ (and that of $\phi$ through $\beta:=\phi^{-1}$), by the regularity theory developed in \cite{CCV}, admits a uniformly bounded solution which will be smooth since the sequence of nondegenerate increasing nonlinearities $\beta_\epsilon$ is smooth. This fact leads to the proof of existence of weak solutions as it is described, for instance, in the proof of Theorem 2.4 of \cite{TEV2} where the penalization for the boundary Stefan problem (see \cite{AC}) is considered.  Actually, when we say that we approximate the problem (\ref{eq3.1}), we mean that we choose a sequence $\{\phi_{\epsilon_n}\} $ of smooth functions such that 
 $$0<\phi_{\epsilon_n}\leq \max \bigg\{\frac{1}{a},\frac{1}{b}\bigg\},  \ \  \phi_{\epsilon_n}(0)=0, \ \  \phi_{\epsilon_n}\rightarrow\phi, \ \  \text{uniformly on} \ \  \R  $$
 where 
 \begin{align}\phi(x)= \begin{cases} \label{eq3.2} 
\frac{x+1}{a}, & x\leq -1  \cr 0, & -1\leq x \leq 1 \cr \frac{x-1}{b}, & x>1  
\end{cases}
\end{align}
and we consider the problem
\begin{align}
\begin{cases} \label{eq3.3}
\partial_t v_{\epsilon_n}-L(\phi_{\epsilon_n}(v_{\epsilon_n}))=0 \ \ & (x,t)\in Q:=\Omega \times (0,T] \\
v_{\epsilon_n}(x,t)=0 \ \ & (x,t)\in \Omega^c\times (0,T)] \\
 v_{\epsilon_n}(x,0)=\beta_{\epsilon_n}(u_0 (x)) \  \  & x\in \Omega\times\{0\}
\end{cases}
\end{align} 
where $\beta_{\epsilon_n}=\phi_{\epsilon_n}^{-1}$. Obviously, problem (\ref{eq3.3}) admits a weak solution and moreover, using the results of \cite{CCV}, this solution is H\"older continuous whose H\"older norm depends on $\epsilon_n$ but its $L^{\infty}$ bound is independent of $\epsilon_n$. Since our result is obtained via apriori estimates, this is enough in order to proceed.

\section{Principal estimate}\label{prest}
We want to prove an oscillation decay for  the approximate solution $u^{\epsilon} $, and we do it in two steps. In this section we are dealing with first one. Our starting point is the approximate problem to (\ref{eq1.1}) i.e.
\begin{eqnarray}\label{eq0}
 \partial_t \beta_\epsilon
 (u^\epsilon(x,t)-Lu^\epsilon(x,t)=0 & \text{in}  & \Omega\times(0,T] \nonumber  \\
u^\epsilon(x,t)=0   & \text{on}  & (\R^n\setminus\Omega)\times(0,T] \\ 
u^\epsilon(x,0)=u_0 (x) & \text{in} & \Omega\times \{0\} \nonumber 
\end{eqnarray}
where $\beta_\epsilon$ is piecewise linear and $u^\epsilon$ is continuous with $L^\infty$ bound independent of $\epsilon$. We seek a priori  bounds independent of $\epsilon$. Our approach is that of DeGiorgi's (see \cite{DG}). 

In order to simplify matters, we normalize the solution i.e. we take $0\leq u^\epsilon\leq1$ in a cylundrical domain and we obtain in the interior of this domain a decay in the oscillation, that is, $\osc{u^\epsilon}<1/2 $.  Notice that $\beta_\epsilon(0)\neq0$ but $\beta_\epsilon$ will be zero at some other point in $[0,1]$ (i.e. at $\frac{-m}{M-n}$ where $M>0$ and $m<0$ are the upper and lower bounds of the unnormalized solution, respectively).
\begin{lemma}\label{mlemma}
Let $Q_1:=B_1\times[-1,0]\subset\Omega\times[-T,T]$. Suppose that $u^\epsilon$ is a solution to (\ref{eq0}) with
$$0\leq u^{\epsilon}\leq 1$$
in $Q_1$, then there exists a constant $\sigma>0$ independent of $\epsilon$ such that 
$$\fint_{Q_1} u^\epsilon dx<\sigma$$
implies
$$u^\epsilon<\frac{1}{2}$$
in $Q_{1/2}$.

\end{lemma}

\begin{proof}
We start by developing the necessary energy inequalities for these equations. We assume that $\beta_\epsilon$ are smooth approximations to $\beta$ satisfying $\beta_{\epsilon}^{'}\geq c_1>0$, and $\beta_\epsilon$ is bounded on $\R$. For simplicity, we drop the  $\epsilon$ subscript. 
Choose a smooth cutoff function $\zeta$ vanishing near the parabolic boundary of $Q_1:=B_1\times (-1,0]$ and $k>0$. Therefore in the weak formulation of (\ref{eq0}) with $\zeta^2(u-k)^+$ as a test function we have
\begin{equation}\label{maineq}\int_{-1}^0\bigg[\int_{\R^n}\partial_t \beta(u(x,t))(\zeta^2(u-k)^+)(x,t)dx+\frac{1}{2} D(u,\zeta^2(u-k)^+)\bigg]dt=0
\end{equation} 
where
$$D(u,v):=\int_{\R^n}\int_{\R^n}(u(x,t)-u(y,t))K(x,y,t)((v(x,t)-v(y,t))dydx.$$
Observe that
\begin{equation}\label{eq1}D(u,\zeta^2(u-k)^+)=D((u-k)^+,\zeta^2(u-k)^+)+D(-(u-k)^-,\zeta^2(u-k)^+).\end{equation}

Multiplying out and rearranging we obtain, for the first term of (\ref{eq1}), 
$$D((u-k)^+,\zeta^2(u-k)^+)=\int_{\R^n} \int_{\R^n}\bigg[(\zeta(u-k)^+)^2(x,t)-(\zeta^2(x,t)+\zeta^2(y,t))(u-k)^+(x,t)(u-k)^+(y,t) $$ 
$$+(\zeta(u-k)^+)^2(y,t)\bigg]K(x,y,t)dydx$$
\begin{equation}\label{eq2}
=D(\zeta(u-k)^+,\zeta(u-k)^+)-\int_{\R^n}\int_{\R^n}(\zeta(x,t)-\zeta(y,t))^2(u-k)^+(x,t)(u-k)^+(y,t)K(x,y,t)dydx\end{equation}
where we have used the identity $a^2+b^2=(a-b)^2+2ab$.

We want to estimate both terms in (\ref{eq2}). For the first term, 
$$D(\zeta(u-k)^+,\zeta(u-k)^+):=\iint\limits_{\R^{2n}}[(\zeta(u-k)^+ )(x,t)-(\zeta(u-k)^+ )(y,t)]^2K(x,y,t)dydx\ \ \ \ \ \ \ \ \ \ \ \ \ \ \ \ \ \ \ \ \ $$
$$ \ \ \ \ \ \ \ \ \ \\ \ \ \ \ \ \ \ \ \ \ \ \geq\frac{1}{\Lambda}\iint\limits_{\R^{2n}}\frac{\big[(\zeta(u-k)^+ )(x,t)
-(\zeta(u-k)^+ )(y,t)\big]^2}{|x-y|^{n+\alpha}}
\mathbf{1}_{\{|x-y|\leq \frac{1}{2}\}}dydx$$
$$\ \ \ \ \ =\frac{1}{\Lambda}\iint\limits_{\R^{2n}}\frac{[(\zeta(u-k)^+ )(x,t)
-(\zeta(u-k)^+ )(y,t)]^2}{|x-y|^{n+\alpha}}dydx$$
$$\ \ \ \ \ \ \ \ \ \ \ \ \ \ \ \ \ \ \ \ \ \ \ \ \ \ \ \ \ \ \ \ \ \ \ \ \ \ \ \      \ \ \ \ \ \ \ \ \ \ -\frac{1}{\Lambda}\iint\limits_{|x-y|>\frac{1}{2}}\frac{[(\zeta(u-k)^+ )(x,t)
-(\zeta(u-k)^+ )(y,t)]^2}{|x-y|^{n+\alpha}}dydx$$  
$$\ \ \ \ \ \ \geq\frac{1}{\Lambda}\iint\limits_{\R^{2n}}\frac{[(\zeta(u-k)^+ )(x,t)-(\zeta(u-k)^+ )(y,t)]^2}{|x-y|^{n+\alpha}}dydx$$  
$$\ \ \ \ \ \ \ \ \ \ \ \ \ \ \ \ \ \ \ \ \ \ \ \ \ \ \ \ \ \ \ \ \ \ \ \ \ \ \ \      \ \ \ \ \ \ \ \ \ \ 
-\frac{2}{\Lambda}\iint\limits_{|x-y|>\frac{1}{2}}\frac{[(\zeta(u-k)^+ )(x,t)]^2+[(\zeta(u-k)^+ )(y,t)]^2}{|x-y|^{n+\alpha}}dydx$$
$$\ \ \ \ \ \ \ \ \ \ \ \ \ \ \ \ \ \ \ \ \ \ \ \ \ \ \ \ \ \ \ \geq\frac{1}{\Lambda}\norm{\zeta(u-k)^+}^2_{H^\frac{\alpha}{2}}-\frac{4}{\Lambda}\int_{\R^n}(\zeta(u-k)^+)^2(x,t)\bigg(\int\limits_{|y-x|>\frac{1}{2}}\frac{1}{|x-y|^{n+\alpha}}dy\bigg)dx$$
$$\ \ \ \ \ \ \ =\frac{1}{\Lambda}\norm{\zeta(u-k)^+}^2_{H^\frac{\alpha}{2}}-\frac{2^{\alpha+2}\omega_n}{\Lambda\alpha }\int_{\R^n}(\zeta(u-k)^+)^2(x,t)dx$$
where in line 4 above we have used the inequality $(a-b)^2\leq 2(a^2+b^2)$ for the second term and on line 5 the definition of fractional Sobolev spaces.

For the second term of (\ref{eq2}) we have
$$\iint\limits_{\R^{2n}}
(\zeta(x,t)-\zeta(y,t))^2(u-k)^+(x,t)(u-k)^+(y,t)K(x,y,t)dydx\ \ \ \ \ \ \ \ \ \ \ \ \ \ \ \ \ \ \ \ \ \ \ \ \ \ \ \ \ \ \ \ \ \ \ \ \ \ $$
$$\leq 2\int\limits_{B_1}(u-k)^+(x,t)\int\limits_{\R^n}(\zeta(x,t)-\zeta(y,t))^2K(x,y,t)dydx\ \ \ \ \ \ \ \ \ \ \ \ \ \ \ \ \ $$
$$\ \ \ \ \ \ \ \ \ \ \ \ \ \ \ \ \ \ \leq 2\Lambda\int\limits_{B_1}(u-k)^+(x,t)\bigg(\int\limits_{|y-x|>\frac{1}{2}}\frac{2}{|x-y|^{n+\alpha}}dy+\int\limits_{|x-y|\leq\frac{1}{2}}
\frac{|\grad\zeta(x+s_0(y-x)|^2}{|x-y|^{n+\alpha}}dy\bigg)dx $$ 
$$\ \ \ \ \ \ \ \ \ \ \ \ \ \ \ \ \ \ \ \ \ \ \ \ \ \ \ \ \ \ \ \ \ \ \ \ \ \ \ \ \ \ \ \ \ \ \ \ \ \ \ \ \ \ \ \ \ \ \ \ \ \ \ \ \ \ \ \ \ \ \ \ \ \ \ \ \ \ \ \ \ \ \ \ \  \ \ \ \ \ \ \ \ \ \ \ \ \ \ \ \ \ \ \ \ \ \text{for} \ \ s_0\in(0,1)$$
$$\leq C_\alpha\Lambda\int\limits_{B_1}(u-k)^+(x,t)dx.\ \ \ \ \ \ \ \ \ \ \ \ \ \ \ \ \ \ \ \ \ \ \ \ \ \ \ \ \ \ \ \ \ \ \ \ \ \ \ \ \ \ \ \ \ \ \ \ \ \ \ \ \ $$

Similarly, the second term of (\ref{eq1}), 
\begin{equation}\label{eq3}D(-(u-k)^-,\zeta^2(u-k)^+)=\ \ \ \ \ \ \ \ \ \ \ \ \ \ \ \ \ \ \ \ \ \ \ \ \ \ \ \ \ \ \ \ \ \ \ \ \ \ \ \ \ \ \ \ \ \ \ \ \ \ \ \ \ \ \ \ \ \ \ \ \ \ \ \ \ \ \ \ \ \ \ \ \ \ \ \ \ $$
$$\ \ \ \ \ \ \ \  =\int_{\R^n}\int_{\R^n}((u-k)^-(x,t)(\zeta^2(u-k)^+)(y,t)+(u-k)^-(y,t)(\zeta^2(u-k)^+)(x,t))K(x,y,t)dydx$$
$$\ \ \ \ \ \ \ \ \ \ \ \ \ 
-\int_{\R^n}\int_{\R^n}(((u-k)^-\zeta^2(u-k)^+)(x,t) +((u-k)^-\zeta^2(u-k)^+)(y,t))K(x,y,t)dydx;\end{equation}
since the second term in (\ref{eq3}) is zero and using the symmetry of the kernel we have
\begin{equation}\label{eq4a} 
D(-(u-k)^-,\zeta^2(u-k)^+)=2\int_{\R^n}\int_{\R^n}((u-k)^-(x,t)(\zeta^2(u-k)^+)(y,t)K(x,y,t)dydx.
\end{equation}
The first term of (\ref{maineq}) becomes
$$\int_{-1}^0\int_{\R^n}\partial_t \beta(u(x,t))(\zeta^2(u-k)^+)(x,t)dxdt=\int_{-1}^0\int_{\R^n}B((u-k)^+)_t\zeta^2dxdt\ \ \ \ \ \ \ \ \ \ \ \ \ \ \ \ \ \ \ \ \ \ \ \ \ \ $$
\begin{equation}\label{eq4}
\ \ \ \ \ \ \ \ \ \ \ \ \ \ \ \ \ \ \ \ \ \ \ \ \ \ \ \ \ \ \ \ \ \ \ \ \ \ \ \ \ \ =\int_{\R^n}(\zeta^2B((u-k)^+))(x,0)dx-\int_{-1}^0\int_{\R^n}B(u-k)^+)(\zeta^2)_tdxdt 
\end{equation}
where
$$B((u-k)^+):=\int_k^u \beta'(s)(s-k)ds=\int_0^{(u-k)^+}\beta'(k+\tau)\tau d\tau.$$
Again, we want to estimate both terms of (\ref{eq4}) using the properties of $\beta$, i.e.
$$B((u-k)^+) \geq c_1 \int_0^{(u-k)^+}\tau d\tau=\frac{c_1}{2}[(u-k)^+]^2$$
and
$$B((u-k)^+)\leq(u-k)^+\int_0^{(u-k)^+}\beta'(l+\tau)d\tau\leq(\beta(1)-\beta(0))(u-k)^+.$$
Thus,
$$\int_{\R^n}(\zeta^2B((u-k)^+))(x,0)dx\geq\frac{c_1}{2}\int_{\R^n}\bigg((u-k)^+\bigg)^2(x,0)dx$$
and
$$\int_{-1}^0\int_{\R^n}B(u-k)^+)(\zeta^2)_tdxdt\leq (\beta(1)-\beta(0))\int_{-1}^0\int_{\R^n}(u-k)^+(\zeta^2)_tdxdt.$$
We substitute all of the above estimates into (\ref{maineq}). The fact that the upper limit $0$ in the t-integration could have been replaced by any $t\in[-1,0]$ yields our energy inequality i.e.
$$\sup_{-1\leq t\leq 0}\int_{\R^n}(\zeta(u-k)^+)^2(x,t)dx+\int_{-1}^0 \|\zeta(u-k)^+\|^2_{H^{\frac{\alpha}{2}}}(t)dt\leq\ \ \ \ \ \ \ \ \ \ \ \ \ \ \ \ \ \ \ \ \ \ \ \ \ \ \ \ \ \ \ \ \ \ \ \ \  $$
\begin{equation}\label{eq5}
\ \ \ \ \ \ \ \ \ \ \ \ \ \ \ \ \ \ \ \ \ \ \ \ \ \ \ \ \ \ \ \ \ \ \ \ \ \ \ \ \ \ \ \ \ \ \ \ \leq CR^{-\alpha}\int_{-1}^0\int_{B_1}\{((u-k)^+)^2+(u-k)^+\}dxdt
\end{equation} 
where $C$ depends only on $n, \Lambda, \alpha$, and $\beta$.

We will obtain now an iterative sequence of inequalities using (\ref{eq5}). 
Thus we define for $m=0,1,2,...$
$$k_m=\frac{1}{2}(1-\frac{1}{2^m}) \ \ \ \ \ \ \ \ R_m=\frac{1}{2}(1+\frac{1}{2^m})  \ \ \ \ \ \ \ \ \ \ \ \ Q_m=\{ (x,t): |x|\leq R_m, -R^\alpha_m\leq t\leq 0 \}$$
and choose cutoff functions  $\zeta_m$ such that 
$$ \mathbf{1}_{Q_{m+1}} \leq \zeta_m \leq \mathbf{1}_{Q_m}, \ \ \ \ \ \ \ \ \ \ \ \ \ |\grad\zeta_m|\leq C2^m, \ \ \ \ |\zeta_mt|\leq C 2^{\alpha m}.$$
We set $u_m:=(u-k_m)^+$ and
$$I_m:=\sup_{-R^\alpha_m\leq t\leq 0}
\int (\zeta_mu_m)^2dx+\int\norm{\zeta_mu_m}^2_{H^{\frac{\alpha}{2}}}dx.$$
By (\ref{eq5})
\begin{equation}\label{eq6}I_m\leq C4^m\bigg(\int (\zeta_{m-1}u_m)^2dxdt+
\int \zeta_{m-1}u_m dxdt\bigg)
\end{equation} 
Since $u_m<u_{m-1}$ and $\{u_m\neq 0\}=\{u_{m-1}>2^{-m-1}\}$
$$\int\zeta_{m-1}u_mdxdt\leq 
 \frac{1}{2}\int(\zeta_{m-1}u_m)^2 dxdt+
\frac{1}{2}|\{u_m\neq 0\}\cap{Q_{m-1}}|$$
$$\leq\frac{1}{2}\int(\zeta_{m-1}u_m)^2 dxdt+\frac{4^{m+1}}{2}\int(\zeta_{m-1}u_{m-1})^2dxdt
\leq\frac{1}{2}(1+4^{m+1})\int(\zeta_{m-1}u_{m-1})^2dxdt.$$
Therefore
$$I_m\leq C4^m\bigg(\int(\zeta_{m-1}u_{m-1})^{2\frac{n+\alpha}{n}}dxdt\bigg)^\frac{n}{n+\alpha}\big|\{u_{m-1}\neq 0 \}\cap {Q_{m-1}}\big|^\frac{\alpha}{n+\alpha}$$
$$\leq C4^{m(1+\frac{\alpha}{n})}\bigg(\int(\zeta_{m-2}u_{m-2})^{2\frac{n+\alpha}{n}}dxdt\bigg).
$$
Hence, by Sobolev inequality, 
\begin{equation}\label{eq7}
I_m\leq C 4^{2m}I_{m-2}^{1+\frac{\alpha}{n}}
\end{equation}
and, consequently, $I_m\rightarrow 0$ as $n\rightarrow \infty$ provided that
$$I_0\leq 4^{-\frac{4n^2}{\alpha^2}} C^{-\frac{n}{\alpha}}.$$
\end{proof}




\section{Oscillation decay}\label{osc-sect}
In order to complete the oscillation decay we have to consider the other alternative i.e. when $u^\epsilon$, on average, is far from the zero of $\beta_\epsilon$. Below, Lemma \ref{altl} handles the second more delicate alternative situation to Lemma \ref{mlemma}. It relies on a parabolic version of DeGiorgi's isoperimetric lemma (Lemma \ref{isolemma}) adjusted to our situation:
\begin{lemma}\label{isolemma}
Given $\sigma>0$ there exists a $\delta>0$ and a $\lambda\in (0,1)$ such that, for a bounded $v$ satisfying  (\ref{eq0}) with $\beta'(v)=1$, $v\geq 0$ in $Q_1$,
and
$$|\{(x,t)\in Q_1: v\geq 1\}|\geq c_0\sigma|Q_1|,\ \ \ \ \text{for some}\ \ c_0<1$$
if
$$\big|\{(x,t)\in Q_1: \l<v<1\}\big|<\delta|Q_1|$$
then
$$\fint_{Q_{1/2}}\bigg(\bigg((1-\frac{v}{\lambda}\bigg)^+\bigg)^2dxdt<\sigma.$$

\end{lemma}
 
We give its proof in the Appendix i.e. \S \ref{App}. 

\begin{lemma}\label{altl}
Let $Q_1$ and $\sigma$ be as in Lemma \ref{mlemma} and 
$$0<u^\epsilon <1  \  \  \  \ \ \ \text{in}  \ \ Q_1$$
a solution of (\ref{eq0})
with $\beta'_{\e}(s)< C$, for $s<1/4$ and $C$ is independent of $\epsilon$. Suppose that
\begin{equation}\label{eq8}
\fint_{Q_1}(u^\epsilon)^2dxdt\geq\sigma,
\end{equation} 
then there exists a small constant $\bar{\s}>0$ depending on $\s$ such that 
$$u^\epsilon\geq \bar{\sigma} \ \ \  \ \ \   \ \text{in} \ \ Q_{1/4}.$$
\end{lemma}
\begin{proof}
By hypothesis, it follows that 
\begin{equation}\label{eq9}
\bigg|\bigg\{u^\epsilon>\frac{\sigma}{2}\bigg\}\cap Q_1\bigg|\geq c_0\sigma|Q_1|
\end{equation} 
for some $c_0<1$ (we drop  the "$\epsilon$" again). Therefore, with $\zeta\in C_0^1(Q_1)$, we set in (\ref{eq0}) $\eta=\zeta^2(1-\frac{2}{\sigma}u)^+$. 
Notice in this range  of values  $\beta'=a$ and with no loss of generality we set $a=1$. So 

\begin{equation}\label{eq10}\int_{-\infty}^0\bigg[\int_{\R^n} \partial_t u(\zeta^2(1-\frac{2}{\sigma}u)^+)dx+\frac{1}{2} D(u,\zeta^2(1-\frac{2}{\sigma}u)^+)\bigg]dt=0.
\end{equation}
Also, (\ref{eq10}) is satisfied if $\eta:=\zeta^2(1-\frac{2}{\l^k\s}u)^+$ for any $\l\in (0,1)$ and $k=0,1,2,...$. Since $\{u>\l^k\frac{\s}{2}\}\supset\{u>\frac{\s}{2}\}$, it follows that, for every $k=0,1,2,...$,  $v_k:=\frac{2}{\l^k\s}u$ satisfies the hypotheses of the above Lemma \ref{isolemma}. 

Let $\d>0$ and $\l\in(0,1)$ be as in Lemma \ref{isolemma}, we will show that, in a finite nunber of steps $k_0:=k_0(\d)$, 
$$ |\{v_{k_0} \leq \l \} \cap Q_1|=0. $$
Indeed, if for $ k=0,1,2,..,k_0 \ \  |\{\l<v_k<1\}\cap Q_1|\geq\d |Q_1|$ then
$$|\{v_k \leq \l\}\cap Q_1|=|\{v_k\leq 1\}\cap Q_1|-|\{\l<v_k\leq 1\}\cap Q_1|$$
$$\ \ \ \ \ \leq |\{v_k\leq 1\}\cap Q_1|-\d |Q_1|$$  
$$\ \ \ \ \ =|\{v_{k-1}\leq\l\}\cap Q_1|-\d | Q_1|$$
$$\ \ \ \ \leq |\{v_0\leq\l\}\cap Q_1|-k\d|Q_1|$$  
$$\ \ \ \ \ \ \ \ \ \ \ \ \ \ \ \ \ \leq 0\ \ \ \ \ \ \ \ \ \ \ \ \ \ \ \ \ \\ \ \\ \ \ \ \ \ \ \ \ \ \ \ \ \ \ \text{if}\ \  k\geq\frac{1}{\d}.$$
Hence, for $k_0:=\frac{1}{\d}$,
$$v_{k_0}> \l$$
that is
$$u> \l^{k_0+1}\frac{\s}{2}.$$

Suppose now that there exists $k^*$, \ $0\leq k^*\leq k_0$, such that
$$\big|\{\l<v_{k^*}\leq 1\}\cap Q_1\big|<\d. $$
By Lemma \ref{isolemma} applied to $v_{k^*}$ and, observing that $w:=(1-v_{k^*+1})$ is a weak solution to (\ref{eq0}) with $\beta'=1$, we can still apply  Lemma \ref{mlemma} to $w=:(1-v_{k^*+1})$  in $Q_{1/2}$ to conclude that 
$$w<\frac{1}{2}\ \ \ \ \ \ \ \text{in}\ \ \ Q_{1/4}$$
or
$$v_{k^*+1}>\frac{1}{2}\ \ \ \ \ \ \ \text{in}\ \ \ Q_{1/4}$$
i.e
$$u>\l^{k^*+1}\frac{\s}{4}\ \ \ \ \ \ \ \text{in}\ \ \ Q_{1/4}.$$
We complete the proof by observing that in both cases we have
$$u>\l^{{k_0}+1}\frac{\s}{4}\ \ \ \ \ \ \ \text{in}\ \ \ Q_{1/4}.$$

\end{proof}

We conclude the section by proving the full oscillation decay of our approximate normalized solution.

\begin{thm}\label{osc}
Let $u^{\e}$ be a solution to (\ref{eq0}) with 
$$\ \ \ \ \ \ \ \ \ \ \ 0<u^\e<1\ \ \ \ \ \ \ \text{in}\\ \ \ Q_1. $$
Then there exists $\bar{\s} $, $0<\bar{\s}<1$,  independent of $\e$, such that
$$\osc_{Q_{1/4}} u^\e\leq 1-\bar{\s}.$$
\end{thm}
\begin{proof}
Suppose first that the singularity  of $\beta_\e$ is greater or equal to one-half i.e. $\beta_\e(x_0)=0$ when $x_0\geq\frac{1}{2}$. Then if $u^\e$ is, in measure, close of order $\s$ to its minimum (i.e. zero) then by Lemma \ref{mlemma}
$$\osc_{Q_{1/2}}u^\e\leq\frac{1}{2};$$ 
if it is far from it, then, since $\beta_\e'(x)<C$, for $x\leq\frac{1}{4}$, where $C$ independent of sufficiently small $\e>0$, Lemma \ref{altl} can be applied to have
$$\osc_{Q_{1/4}} u^\e\leq 1-\bar{\s}$$
where $\bar{\s}:=\l^{k_0+1}\frac{\s}{2}$.

On the other hand, if the singularity of $\beta_\e$ is less than one-half then we argue in a similar fashion, i.e. if $u^\e$ is, in measure, close of order $\s$ to its maximum (i.e. one) then by Lemma \ref{mlemma} applied to $1-u^\e$ we obtain the same bound for the oscillation as above; on the other hand, if it is far from its maximum, we apply Lemma \ref{altl} to $1-u^\e$ again.

\end{proof}

\section{Continuity}\label{cont}
In this section we are ready to prove our main result. This is achieved by iterating the Theorem \ref{osc} of the previous section. The iteration will be carried out in a dyadic sequence of shrinking cylinders which will force the oscillations to diminish. Since the estimates will deteriorate and $\beta'_\epsilon$ goes to infinity, our modulus of continuity will not be H\"older, except in case (ii) where we have an extrarescaling invariance.

\begin{prop}\label{prop1}
Let $u^\epsilon$ be a solution to problem (\ref{eq0}) in $Q_R(x_0,t_0)\subset Q$ with $R\leq 1$  Suppose that 
$$\beta_\epsilon \big(\sup_{Q_R} u^\epsilon \big)
-\beta_\epsilon \big(\inf_{Q_R} u^\epsilon \big) 
\leq K \ \ \ \ \text{and} \ \ \ \ \inf_{Q_R} 
\beta'_\epsilon\geq c_1>0 $$
where $K$ and $c_1$ are independent of $\epsilon$. Then
$$|u^\epsilon(x,t)-u^\epsilon(x_0,t_0)|\leq \omega(x-x_0,t-t_0)$$
where $\omega$ is a modulus of continuity (i.e. a monotone function with $\omega(0)=0$ depending only on $K$ and $c_1$.
\end{prop}
\begin{proof}
We ignore "$\epsilon$" again. Set 
$Q_k(x_0,t_0):=Q_{R/{2^k}}(x_0,t_0)$ and $m_k:=\inf_{Q_k} u$, $M_k:=\sup_{Q_k} u$. Define
$$v:=\frac{u_k-m_k}{M_k-m_k}$$
where $u_k (x,t):=u(x_0+R\frac{x}{2^k},t_0+R^\a\frac{t}{2^{\alpha k}})$. Then $v$ verifies
$$\tilde{\beta'}(v)v_t=Lv \ \ \ \text{in} \ \ \ Q_1 $$
where $\tilde{\beta}(v):=\frac{1}{M_k-m_k}\beta((M_k-m_k)v+m_k)$. The function $v$ still is a solution to the equation of (\ref{eq0}) but with a different kernel. Actually this new kernel $\bar{K}(x,y,t)$ satisfies the same conditions imposed upon $K(x,y,t)$ in a sronger sense.  Indeed, simce $\Bar{K}(x,y,t):=\big(\frac{R}{2^k}\big)^{n+\a}K(x_0+\frac{R}{2^k}x,x_o+\frac{R}{2^k}y, t_0+\big(\frac{R}{2^k}\big)^\a t)$ we have 
$$\frac{{\bf{1}}_{\{|x-y|\leq (2^{k+1}/R)\}} }{\L}\frac{1}{|x-y|^{n+\a}}\leq\bar{K}(x,y,t)\leq\frac{\L}{|x-y|^{n+\a}} $$   
Hence we can apply Theorem \ref{osc} to $v$ to obtain
$$\osc_{Q_{1/4}} v\leq (1-\bar{\sigma})$$
where the direct dependence of the constant $\bar{\s}$ to the nonlinearity of $\beta$ is given by 
$$\bar{\sigma}=\bar{\s}\bigg(\frac{\inf_{Q_R} \tilde{\beta}'(v)}{\tilde{\beta}(1)-\tilde{\beta}(0)}\bigg):=C\bigg(\frac{\inf_{Q_R} \tilde{\beta}'(v)}{\tilde{\beta}(1)-\tilde{\beta}(0)}\bigg)^{N_0}, $$where $C$ and $N_0$ are universal i.e. they are independent of $\e$ and depend only on $\L$, $\a$, the dimension $n$ and the $L^\infty$ bounds of the solution. 
Hence, in the original setting,
$$\osc_{Q_{k+1}(x_0,t_0)} u\leq\mu_k \osc_{Q_R(x_0,t_0)} u$$
with  $\mu_k:=1-\bar{\s}(\frac{c_1}{K}\osc_{Q_k(x_0,t_0)} u)$. Therefore $\mu_k \rightarrow 1$ as $k\rightarrow \infty$ only when $\osc_{Q_k} u\rightarrow 0$ which yields a logarithmic raised to small power modulus of continuity.
\end{proof}
The additional assumption that we will impose on $\beta$ of case (ii) which includes the porous media situation i.e. $\beta(u) \sim u^{1/m}, m>1$ will give us H\"older continuity.
\begin{prop}\label{prop 4.6}
Let $u^\epsilon$ be a solution to problem (\ref{eq0}) in $Q_R(x_0,t_0)\subset Q$, $R\leq 1$, with $\beta_\epsilon$ being as the one in case (ii). Suppose that for any $m<M$ 
$$\frac{\inf_{[m,M]}\beta'_\epsilon\cdot(M-m)}{\beta_\epsilon(M)-\beta_\epsilon(m)}\geq l$$
where $l$ is a very small positive constant independent of $\epsilon$. Then
$$|u^\epsilon(x,t)-u^\epsilon(x_0,t_0)|\leq \omega(x-x_0,t-t_0)$$
where $\omega$ is a H\"{o}lder modulus of continuity with exponent depending only on $l$.
\end{prop}
\begin{proof}
As in the proof of the Threorem \ref{osc} we arrive at 
$$\osc_{Q_{k+1}(x_0,t_0)} u^{\e}\leq(1-Cl^{N_0}) \osc_{Q_k(x_0t,t_0)} u^{\e}$$
or
$$\osc_{Q_{k+1}(x_0,t_0)} u^\e\leq(1-Cl^{N_0})^k \osc_{Q_R}(x_0,t_0) u$$
where the universal constants $C$ and $N_0$ are the ones in the proof of the previous Proposition \ref{prop1}. The H\"older continuity of $u^\epsilon$ follows easily.
\end{proof}

The next Theorem is the main result of the present paper.

\begin{thm}\label{mth}
Let $u$ be a solution to (\ref{eq1.1}) with $\beta$ satisfying (\ref{eq1.2}) or (\ref{eq1.3}). Then $u$ is continuous with a modulus of continuity that depends on the nature of the singularity of $\beta$.
\end{thm}
\begin{proof}
By standard methods and using Proposition \ref{prop1} or Proposition \ref{prop 4.6}, we can extract a subsequence $u^{\epsilon_m}$ that converges uniformly to our solution $u$. The nature of the modulus of continuity is depending whether we are imposing  condition (\ref{eq1.2}) or (\ref{eq1.3}).  
\end{proof}
\section{Appendix}\label{App}
Now we give the proof of Lemma \ref{isolemma} (see also Lemma 4.1 of \cite{CCV})
\begin{proof}
By contradiction, if not true then 
$$0<\s\leq\fint_{Q_{1/2}}\bigg(\big(1-\frac{v}{\l}\big)^+\bigg)^2 dxdt\leq\frac{|Q_{1/2}\cap\{\l>v\}|}{|Q_{1/2}|} $$
So, there exists a $t_0>-\frac{1}{2}$ such that 
$$|\{x\in B_{1/2}:(\l-v))(x,t_0)>0\}|>\frac{\s}{2}|Q_{1/2}|.$$
Therefore for $\mathbf{1}_{B_1}\leq\xi(x)\leq\mathbf{1}_{B_2}$ and for  $\l'=2\l$
$$\int_{\R^n}\big[\xi\big(\l'-v\big)^+\big]^2dx\geq (\l'-\l)^2|B_{1/2}\cap \{\l>v\}|\geq(\l')^2\frac{\s}{8}|Q_{1/2}|    $$
i.e.
\begin{equation}\label{ineq10}
E(t_0):=\int_{\R^n} \big(\xi(x)(\l'-v))^+(x,t_0)\big)^2dx\geq(\l')^2\frac{\s}{8}|Q_{1/2}|.
\end{equation}

On the other hand, since $v$ is a solution to (\ref{eq0}) with $\beta'_\epsilon(v)=1$, taking in its weak formulation $\eta(x,t);=\big(\zeta^2(\l'-v)^+\big)(x,t)$ with $\zeta(x,t)=\xi(x)\tau(t)$ such that $\mathbf{1}_{[-1,0}\leq\tau (t)\leq\mathbf{1}_{[-2,0]}$ and keeping the same $\xi$ as above i.e. $\mathbf{1}_{B_{1}}\leq \xi(x)\leq\mathbf{1}_{B_2}$ we have
\begin{equation}\label{eq7.2}
\int_{-\infty}^0\bigg[\int_{\R^n} (\partial_t u)(\zeta^2(\l'-v)^+)dx+\frac{1}{2} D(u,\zeta^2(\l'-v)^+)\bigg]dt=0.
\end{equation}
Then, as in the proof of Lemma \ref{mlemma}, we obtain
\begin{equation}\label{eq11}
\int\limits_{\R^n} (\zeta (\l'-v)^+)^2(x,0)dx+\int\limits_{-\infty}^0\bigg[D(\zeta (\l'-v)^+ ,\zeta (\l'-v)^+)\ \ \ \ \ \ \ \ \ \ \ \ \ \ \ \ \ \ \ \ \ \ \ \ \ \ \ \ \ \ \ \\ \ \ \ \ \ \ \ \ \ \ \ \ \ \ \ \ \ \ \ \ \ $$ 
$$\ \ \ \ \ \ \ \ \ \ \ \ \ \ \ \ \ \ \ \ \ \ \ \ \ \ \ \ \ \ \ \ \ \ \ \  +2\iint\limits_{\R^{2n}}(\l'-v)^-(x,t)(\zeta^2(\l'-v)^+)(y,t)K(x,y,t)dydx\bigg]dt$$
$$=\int\limits_{-\infty}^0\iint\limits_{\R^{2n}}(\zeta(x,t)-\zeta(y,t))^2 (\l'-v)^+(x,t) (\l'-v)^+(y,t)K(x,y,t)dydxdt\ \ \ \ \ \ \ \ \ \ \ \ \ \ \ \ \ \ \ \ \ \ \ \ \ \ \ \ \ \ \ \ \ \ \ \ \ \ \ \ \ \ $$  
$$\ \ \ \ \ \ \ \ \ \ \ \ \ \ \ \ \ \ \ \ \ \ \ \ \ \ \ \ \ \ \ \ \ \ \ \ \ \ \ \ \ \ \ \ \ \ \ \ \ \ \ \ \ \ \ \ \ \ \ \ \ \ \ \ \ \ \ \ \ \ \ \ \ \ \ \ \ \ \ \ \ \ \ \ \ \ \ \ \ \ \ \ \ \ \ \ \ \ \ \ +\int\limits_{-\infty}^0\int\limits_{\R^n}((\l'-v)^+)^2\zeta^2_tdxdt.
\end{equation}
Notice that the right hand of (\ref{eq11}) is controlled by $C_\alpha(\l')^2$
and the first two terms on the left hand side of (\ref{eq11}) are non-negative therefore, since $\tau(t)\equiv 1$ for $t\in[-1,0]$ we have
\begin{equation}\label{ineq11}
\int\limits_{-1}^{0}\iint\limits_{\R^{2n}}(\l'-v)^-(x,t)(\xi(y))^2(\l'-v)^+)(y,t)K(x,y,t)dydxdt\leq C_\alpha(\l')^2.
\end{equation}
By hypothesis 
$|\{(x,t)\in Q_1: v^+=0\}|\geq c_0\sigma|Q_1|$ and if we set $I;=\{t\in(-1,0); |\{v(.,t)\leq 0\}\cap B_{1}|\geq\frac{1}{2}c_0\s|Q_1|$ then by the fact that $\inf_{|x-y|\geq 2} K(x,y,t)\geq C\Lambda^{-1}$, the left hand side of (\ref{ineq11}) is bounded below by

$$\frac{C(1-\lambda')c_0\sigma[Q_1|}{2\Lambda}\int_I\int_{\R^n}(\xi^2 (\l'-v)^+)(y,t)dydt.$$
Therefore we deduce from (\ref{ineq11}) that 
$$\int_I\int_{\R^n}(\xi (\l'-v)^+)^2(y,t)dydt\leq\frac{C(\alpha,\L)}{c_0\s|Q_1|}(\l')^3$$
where we took $\l'< \frac{1}{2}$ and used that $(\l'-v)^+\leq\l'$. Furthermore, if $(\l')^{1/8}\leq\frac{c_0\s|Q_1|}{C(\alpha,\L)}$ then
$$\int_I\int_{\R^n}(\xi (\l'-v)^+)^2(y,t)dydt\leq (\l')^{3-\frac{1}{8}}.$$
Since $|I|\geq \frac{c_0\s |Q_{1}|}{2|B_1|}$ by choosing a set $F\subset I$ with $|F|<(\l')^{\frac{1}{8}}$ we have
$$\ \ \ \ \ \ \ \ \ \ \ \ \ \ \ \ \ \ \ \ \ \ \int_{\R^n}(\xi (\l'-v)^+)^2(y,t)dy\leq (\l')^{3-\frac{1}{4}}\ \ \ \ \ \ \ \ \ \ \ \ \ \forall\ t \in I-F.$$
Moreover for $(\l')^{1-\frac{1}{4}}\leq\frac{\s}{16}|Q_{1/2}|$ and $\forall\ t\in I-F$
$$E(t)=\int_{5\R^n}\big(\xi(x)(\l'-v)^+(x,t))\big)^2dx\leq (\l')^2\frac{\s}{16}|Q_{1/2}|  $$
and let $t^*<t_0$ ($t_0$ as above) such that $t^* \in I-F$ be the first time for which $E(t^*)\leq (\l')^2\frac{\s}{16}|Q_{1/2}|$. This follows from the fact $\frac{d}{dt}E(t)\leq C'(\l')^2$.

Now, set
$$J:=\big\{(t^*,t_0):\frac{\s}{16}|Q_{1/2}|(\l')^2< E)t)< \frac{\s}{8}|Q_{1/2}|(\l')^2\big\} $$
then,  for $t\in J\cap N$  where $N:=\{t\in [-1,0]:|\{v(.,t)\leq 0\}\cap B_{1}|\geq\frac{1}{2}c_0\s|Q_1|\}$
$$C(\l')^2\geq \int\limits_{-1}^{0}\iint\limits_{\R^{2n}}(\l'-v)^-(x,t)(\xi(y))^2(\l'-v)^+)(y,t)K(x,y,t)dydxdt\ \ \ \ \ \ \ \ \ \ \ $$
$$\ \ \ \ \ \ \ \ \  \ \geq\frac{Cc_0\s|Q_1|}{2\L}\int\limits_{J\cap N}\int\limits_{\R^n}\xi(x)(\l'-v)^+(x,t)dxdt\geq\frac{Cc_0\s|Q_1|}{4\l'\L}\int\limits_{J\cap N}\int\limits_{\R^n}\big(\xi(x)(\l'-v)^+(x.t)\big)^2dxdt $$
$$\geq\frac{C c_0 \s |Q_1|}{4\l'\L}\int\limits_{J\cap N} E(t)dt\geq \frac{Cc_0\s^2\l'|Q_1||Q_{1/2}||J\cap N|}{64\L}.\ \ \ \ \ \ \ \ \ \ \ \ \ \ \ \ \ \ \ \ \ \ \ \ \ \ $$
Therefore
$$|J\cap N|\leq\bar{C}\frac{\l'}{c_0\s^2}  $$
which implies that  for $\l'\leq (c_0\s^2|J|/2\bar{C})$
$$|J\cap N|\leq\frac{|J|}{2}.$$
Thus, for every $t\in J\setminus N$,$$M(t):=|\{\l\leq v(.,t)\leq 1\}|\geq (1-\frac{1}{2}c_0\s|Q_1|-\frac{1}{2}\s|Q_{1/2}|)>\frac{1}{2}.$$
and
$$\big|\{(x,t)\in Q_1: \l<v<1\}\big|\geq\int_{-1}^0M(t)dt\geq\int_{J\setminus N}M(t)dt\geq\frac{|J|}{4}\geq\frac{\s}{16}|Q_{1/2}|.$$
\end{proof}

\section*{Acknowledgments}
I. Athanasopoulos was supported by the University of Cyprus research funds. L. Caffarelli was supported by NSF grant DMS-2108682. The work of E. Milakis was co-funded by the European Regional Development Fund and the Republic of Cyprus through the Research and Innovation Foundation (Project: EXCELLENCE/1216/0025).

\bibliographystyle{plain}   
\bibliography{biblio}             
\index{Bibliography@\emph{Bibliography}}%

\vspace{2em}

\begin{tabular}{l}
Ioannis Athanasopoulos\\ University of Crete \\ Department of Mathematics  \\ 71409 \\
Heraklion, Crete GREECE
\\ {\small \tt athan@uoc.gr}
\end{tabular}
\begin{tabular}{l}
Luis Caffarelli\\ University of Texas \\ Department of Mathematics  \\ TX 78712\\
Austin, USA
\\ {\small \tt caffarel@math.utexas.edu}
\end{tabular}
\begin{tabular}{lr}
Emmanouil Milakis\\ University of Cyprus \\ Department of Mathematics \& Statistics \\ P.O. Box 20537\\
Nicosia, CY- 1678 CYPRUS
\\ {\small \tt emilakis@ucy.ac.cy}
\end{tabular}

\end{document}